                     \def\version{April 6, 2011}                   %

\documentclass[reqno,11pt]{amsart} 
\usepackage{srcltx}

\usepackage{amsmath, amsthm, a4, latexsym, amssymb}

\def\@rmrk#1#2{\refstepcounter
    {#1}\@ifnextchar[{\@yrmrk{#1}{#2}}{\@xrmrk{#1}{#2}}}

\newcommand{\smfrac}[2]{{\textstyle{\frac {#1}{#2}}}}

\makeatletter\@addtoreset{equation}{section}\makeatother

 \newfont{\bfit}{cmbxti10 scaled 1200}

\renewcommand{\d}{{\rm d}}
 \newcommand{\e}{{\rm e} }

 \newcommand{\eps}{\varepsilon}
 \newcommand{\supp}{{\rm supp}}

 \newcommand{\R}{\mathbb{R}}
 \newcommand{\N}{\mathbb{N}}
 \newcommand{\Z}{\mathbb{Z}}

 \newcommand{\me}{\mathbb{E}}
 \newcommand{\E}{\mathbb{E}}
 \newcommand{\om}{\omega}
 \renewcommand{\P}{\mathbb{P}}
 \def\1{{\mathchoice {1\mskip-4mu\mathrm l} 
{1\mskip-4mu\mathrm l}
{1\mskip-4.5mu\mathrm l} {1\mskip-5mu\mathrm l}}}

 \newcommand{\Fcal}{{\mathcal F}}

 \newcommand{\Mcal}{{\mathcal M}}

 \newcommand{\xy}{{xy}}

\newcommand{\ssup}[1] {{\scriptscriptstyle{({#1}})}}

\renewcommand{\subsection}{\secdef \subsct\sbsect}
\newcommand{\subsct}[2][default]{\refstepcounter{subsection}
\vspace{0.15cm}
{\flushleft\bf \arabic{section}.\arabic{subsection}~\bf #1  }
\nopagebreak\nopagebreak}
\newcommand{\sbsect}[1]{\vspace{0.1cm}\noindent
{\bf #1}\vspace{0.1cm}}

{\nopagebreak {\hfill\rule{2mm}{2mm}}\\ }

\newtheorem{theorem}{Theorem}[section]
\newtheorem{lemma}[theorem]{Lemma}
\newtheorem{cor}[theorem]{Corollary}

\newtheorem{remark}[theorem]{Remark}

\newtheoremstyle{thm}{1.5ex}{1.5ex}{\itshape\rmfamily}{}
{\bfseries\rmfamily}{}{2ex}{}

\newtheoremstyle{rem}{1.3ex}{1.3ex}{\rmfamily}{}
{\itshape\rmfamily}{}{1.5ex}{}
\theoremstyle{rem}
\refstepcounter{subsubsection}

\def\thebibliography#1{\section*{References}
  \list%
  {\arabic{enumi}.}
    {\settowidth\labelwidth{[#1]}\leftmargin\labelwidth
    \advance\leftmargin\labelsep
    \parsep0pt\itemsep0pt
    \usecounter{enumi}}
    \def\newblock{\hskip .11em plus .33em minus .07em}
    \sloppy                   
    \sfcode`\.=1000\relax}

\begin{document}
\title[Large deviations for RWRC]
{\large Large deviations for the local times\\ \medskip of a random walk among random conductances}
\author[Wolfgang K\"onig, Michele Salvi and Tilman Wolff]{}
\maketitle
\thispagestyle{empty}
\vspace{-0.5cm}

\centerline{\sc By Wolfgang K\"onig$ $\footnotemark[1]$^,$\footnotemark[2], Michele Salvi\footnote{Institute for Mathematics, TU Berlin, Str.~des 17.~Juni 136, 10623 Berlin, Germany, {\tt koenig@math.tu-berlin.de} and {\tt salvi@math.tu-berlin.de}}
and Tilman Wolff\footnote{Weierstrass Institute Berlin, Mohrenstr.~39, 10117 Berlin, {\tt koenig@wias-berlin.de} and \tt wolff@wias-berlin.de}}
\renewcommand{\thefootnote}{}


\vspace{-0.5cm}
\centerline{\textit{Weierstrass Institute Berlin and TU Berlin}}
\vspace{0.2cm}

\begin{center}
\version
\end{center}

\begin{quote}{\small {\bf Abstract: } We derive an annealed large deviation principle for the normalised local times of a continuous-time random walk among random conductances in a finite domain in $\Z^d$ in the spirit of Donsker-Varadhan \cite{DV75}. We work in the interesting case that the conductances may assume arbitrarily small values. Thus, the underlying picture of the principle is a joint strategy of small values of the conductances and large holding times of the walk. The speed and the rate function of our principle are explicit in terms of the lower tails of the conductance distribution. As an application, we identify the logarithmic asymptotics of the lower tails of the principal eigenvalue of the randomly perturbed negative Laplace operator in the domain.}
\end{quote}

\section{Introduction}

We introduce the main object of our study in Section~\ref{sec-LapOmega}, present our main results in Section~\ref{sec:MainRes} and give a heuristic explanation in Section~\ref{sec:Heur}. The proof of the main theorem is carried out in Sections~\ref{sec:proof_lb} and~\ref{sec:proof_ub}.

\subsection{Continuous-time random walk among random conductances}\label{sec-LapOmega}

\noindent Consider the lattice $\Z^d$ with $E=\{\{x,y\}\colon x,y\in \Z^d, x\sim y\}$ the set of nearest-neighbour bonds. Assign to any edge $\{x,y\}\in E$  a random weight $\om_{\{x,y\}}\in[0,\infty)$. We will use the notation $\om_{\{x,y\}}=\om_\xy=\om_{yx}$ for convenience. Assume that $\om=(\om_\xy)_{\{x,y\}\in E}$ is a family of nonnegative i.i.d.\ random variables. We refer to them as {\it random conductances}. One of the main objects of the present paper is the randomly perturbed Laplacian $\Delta^{\om}$ defined by
\begin{equation}\label{eqn:generator}
\Delta^\omega f(x):=\sum_{y\in\Z^d\colon y\sim x}\omega_\xy (f(y)-f(x)),\qquad f\colon\Z^d\to\R,\,x\in\Z^d.
\end{equation}
This operator is symmetric and generates the continuous-time random walk $(X_t)_{t\in[0,\infty)}$ in $\Z^d$, the {\it random walk among random conductances (RWRC)} or, as many authors call it, {\it random conductance model (RCM)}. This process starts at $x\in\Z^d$ under $\P_x^\om$ and evolves as follows. When located at $y$, it waits an exponential random time with parameter $\sum_{z\sim y}\om_{yz}$ (i.e., with expectation $1/\sum_{z\sim y}\om_{yz}$) and then jumps to a neighbouring site $z'$ with probability $\om_{yz'}/\sum_{z\sim y}\om_{yz}$. We write $\Pr$ for the probability and $\langle\cdot\rangle$ for the expectation with respect to $\om$.

In some recent publications (see, e.g., \cite{BD10}), the above walk is called {\it variable-speed random walk (VSRW)} in contrast to the {\it constant-speed random walk (CSRW)}, where the holding times have parameter one, and to the discrete-time version of the RWRC, where the jumps occur at integer times. Substantial differences between these two variants appear, for example, in slow-down phenomena. These are typically due to extremely large holding times in the former case, but to so-called traps (regions of transition probabilities in which the path loses much time) in the two latter cases. A further aspect is that continuous-time random walks may reach any point in finite time with positive probability, in contrast to discrete-time walks. All these processes are versions of RWRC.

Let us mention some earlier work on RWRC. For the discrete-time setting, a quenched functional CLT is derived in \cite{BP07}, assuming that the conductances take values in $[0,1]$. In \cite{BBHK08} and \cite{FM06}, the authors examine the probability for the random walk to return to the origin in the quenched and annealed case, respectively. Here, the lower tails of the distribution of the conductances have polynomial decay. The quenched functional CLT has been addressed for the CSRW in \cite{M08} and for both the CSRW and VSRW in \cite{BD10}, the former considering conductances in $[0,1]$, the latter requiring the conductances to be bounded away from zero. Weak convergence to some L\'evy process after proper rescaling is established in \cite{BC10} for conductances bounded away from zero.

The main purpose of this paper is the description of the long-time behaviour of the walk in a given finite connected set $B\subset\Z^d$ containing the starting point. More precisely, we derive a {\it large deviation principle (LDP)} for the local times of the walk, which are defined by
\begin{equation}
\ell_t(z)=\int_{0}^t\delta_{X_s}(z)\,\d s,\qquad z\in\Z^d,t>0.
\end{equation}
In words, $\ell_t(z)$ is the amount of time that the walker spends in $z$ by time $t$. The speed and the rate function of this LDP are explicit. 

One application is a characterization of the logarithmic asymptotics of the non-exit probability from $B$. As this is standard and well-known under the quenched law $\P_0^\om$, we will work under the annealed law $\langle\P_0^\om(\cdot)\rangle$ instead. One of our motivations are the seminal works \cite{DV75} and \cite{G77} on large deviations for the occupation time measures of various types of Markov processes. Another one is the question of the extremal behaviour of the principal eigenvalue of the random operator $\Delta^\om$ in $B$. 

We concentrate on the interesting case where the conductances are positive, but can assume arbitrarily small values. Here the annealed behaviour comes from a combined strategy of the conductances and the walk, and the description of their interplay is the focus of our study. Losely speaking, the optimal joint strategy of the conductances and the walk to meet the non-exit condition $X_{[0,t]}\subset B$ for large $t$ is that the conductances assume extremely small $t$-dependent values and the walker realizes very large $t$-dependent holding times and/or trajectories that do not leave $B$. We will informally describe this picture in greater detail.

\subsection{Main result}\label{sec:MainRes}

\noindent Our main assumption on the i.i.d.~field $\om$ of conductances is that, for any $\{x,y\}\in E$,
\begin{equation}\label{mainass}
\omega_\xy\in(0,\infty)\qquad\mbox{and}\qquad{\rm essinf}\;(\omega_\xy)=0.
\end{equation}
More specifically, we require some regularity of the lower tails, namely the existence of two parameters $\eta,D\in(0,\infty)$ such that
\begin{equation}\label{eqn:tails_behaviour}
\log \Pr(\om_\xy\leq \varepsilon)\sim -D\eps^{-\eta},\qquad \eps\downarrow 0.
\end{equation}
That is, the edge weights can attain arbitrarily small values with prescribed probabilities.

Our main theorem is the following large deviation principle for the normalised local times before exiting  $B$. That is, we restrict to the event $\{X_{[0,t]}\subset B\}=\{\supp(\ell_t)\subset B\}$. By
\begin{equation}\label{EBdef}
E_B:=\{\{x,y\}\colon x\in B, y\in\Z^d,y\sim x\}
\end{equation}
we denote the set of edges connecting the sites of $B$ with their neighbours both in $B$ and outside.

\begin{theorem}[Annealed LDP for $\frac 1t \ell_t$]\label{thm:LDP_finite}
Assume that $\om$ satisfies \eqref{mainass} and \eqref{eqn:tails_behaviour}. Fix a finite connected set $B\subset\Z^d$ containing the origin. Then the  process of normalized local times, $(\frac{1}{t}\ell_t)_{t>0}$, under the annealed sub-probability law $\langle\P_0^\om(\,\cdot\,\cap\{X_{[0,t]}\subset B\})\rangle$ satisfies an LDP on $\Mcal_1(B)$, the space of probability measures on $B$, with speed $t^\frac{\eta}{\eta+1}$ and rate function $J$ given by
\begin{equation}\label{eqn:ratefunction_finite}
J(g^2):=K_{\eta,D}\sum_{\{x,y\}\in E_B}|g(y)-g(x)|^\frac{2\eta}{\eta+1},\qquad g\in \ell^2(\Z^d), \supp(g)\subset B, \|g\|_2=1,
\end{equation}
where $K_{\eta,D}=\big(1+\frac{1}{\eta}\big)(D\eta)^\frac{1}{\eta+1}$.
\end{theorem}

The proof of Theorem~\ref{thm:LDP_finite} is given in Section~\ref{sec:proof}. More explicitly, it says
\begin{eqnarray}
\liminf_{t\to\infty}t^{-\frac{\eta}{\eta+1}}\log\Big\langle\P_0^{\om}\Big(\smfrac{1}{t}\ell_t\in O,X_{[0,t]}\subset B\Big)\Big\rangle
&\geq&-\inf_{g^2\in O}J(g^2) \qquad \mbox{for }O\subset \Mcal_1(B)\mbox{ open,} \label{eqn:lbound}\\
\limsup_{t\to\infty}t^{-\frac{\eta}{\eta+1}}\log \Big\langle\P_0^{\omega}\Big(\smfrac{1}{t}\ell_t\in C,X_{[0,t]}\subset B\Big)\Big\rangle
&\leq&-\inf_{g^2\in C}J(g^2)\qquad\mbox{for }C\subset \Mcal_1(B) \mbox{ closed,} \label{eqn:ubound}
\end{eqnarray}
and that the rate function $J$ has compact level sets. Our convention is to extend any probability measure on $B$ trivially to a probability measure on $\Z^d$; note the zero boundary condition in $B$ that is induced in this way.

A heuristic explanation of the speed and rate function is given in Section~\ref{sec:Heur}. It turns out there that the conductances that give the most contribution to the LDP are of order $t^{-1/(1+\eta)}$ and assume a certain deterministic shape.

With the special choice $O=C=\Mcal_1(B)$, we obtain the following corollary.
\begin{cor}[Non-exit probability from $B$]
Under the assumptions of Theorem~\ref{thm:LDP_finite},
\begin{equation}\label{exitasy}
\lim_{t\to\infty}t^{-\frac{\eta}{\eta+1}}\log\Big\langle\P_0^{\om}\big(X_{[0,t]}\subset B\big)\Big\rangle=-K_{\eta,D}L_\eta(B),
\end{equation}
where 
\begin{equation}\label{kappadef}
L_\eta(B)=\inf_{g^2\in \Mcal_1(B)}\sum_{\{x,y\}\in E_B}|g(y)-g(x)|^\frac{2\eta}{\eta+1}.
\end{equation} 
\end{cor}

From Theorem~\ref{thm:LDP_finite}, we also derive the precise logarithmic lower tails of the principal (i.e., smallest) eigenvalue $\lambda^\om(B)$ of $-\Delta^\om$ in $B$ with zero boundary condition.

\begin{cor}[Lower tails for the bottom of the spectrum of $\Delta^\om$]\label{cor-lowtailslambda}
Under the assumptions of Theorem~\ref{thm:LDP_finite},
$$
\lim_{\eps\downarrow 0}\eps^{\eta}\log\Pr(\lambda^\om(B)\leq \eps)=-DL_\eta(B)^{\eta+1}.
$$
\end{cor}

\begin{proof}
A Fourier expansion shows that, $\Pr\text{-almost surely,}$
$$
\P^\om_0(X_{[0,t]}\subset B)=\sum_{i=1}^{|B|}\e^{-t\lambda_i^\om}v_i^\om(0)( v_i^\om,\1)\leq \sum_{i=1}^{|B|}\e^{-t\lambda_i^\om}|B|\leq |B|^2\e^{-t\lambda^\om(B)},
$$
where $0<\lambda^\om(B)=\lambda_1^\om\leq\dots\leq\lambda_{|B|}^\om$ are the eigenvalues of $\Delta^\om$ with zero boundary condition in $B$ and $(v_i^\om)_{i=1,\dots,|B|}$ a corresponding orthonormal base of eigenvectors. We also have, $\Pr\text{-almost surely,}$
$$
\e^{-t\lambda^\om(B)}\leq \sum_{i=1}^{|B|}\e^{-t\lambda_i^\om}( v_i^\om,\1)^2\leq \sum_{z\in B}\P^\om_z(X_{[0,t]}\subset B).
$$
Applying Theorem~\ref{thm:LDP_finite} to $B-z$ and using the shift-invariance of $\omega$, we see that the expectation of the right-hand side has the same logarithmic asymptotics as $\langle \P^\om_0(X_{[0,t]}\subset B)\rangle$.
Therefore, the two above inequalities show that 
\begin{equation}\label{eqn:EVubound}
\log \Big\langle\e^{-t\lambda^\om(B)}\Big\rangle\sim\log \Big\langle\P^\om_0(X_{[0,t]}\subset B)\Big\rangle,\qquad t\to\infty.
\end{equation}
Now de Bruijn's exponential Tauberian theorem \cite[Theorem 4.12.9]{BGT89}, together with \eqref{exitasy} yields the desired asymptotics.
\end{proof}

Theorem~\ref{thm:LDP_finite} holds literally true if $\Z^d$ is replaced by an (infinite or finite) graph and $B$ by some finite subgraph. In future work we will be interested in extensions of Theorem~\ref{thm:LDP_finite} to $B\subset\Z^d$ a $t$-dependent centred box and $\Delta^\om$ replaced by $\Delta^\om +\xi$ with $\xi=(\xi(z))_{z\in\Z^d}$ an i.i.d.~random potential, independent of $\om$.

\subsection{Heuristic derivation}\label{sec:Heur}

We now give a formal derivation of the LDP in Theorem~\ref{thm:LDP_finite}. Given a fixed realisation $\varphi=\{\varphi_\xy\colon\{x,y\}\in E_B\}\in(0,\infty)^{E_B}$ of the conductances, the probability that the normalised local time resembles some realisation  $g^2\in\Mcal_1(B)$ is roughly
\begin{equation}\label{eqn:LD_inthebox_finite}
\P_0^\varphi\Big(\smfrac{1}{t}\ell_t\approx g^2\Big)\approx\exp\big\{-t I_\varphi(g^2)\big\},
\end{equation}
where the corresponding Donsker-Varadhan rate function is given by
\begin{equation}\label{Idef}
I_\varphi(g^2)=\big(-\Delta^\varphi g,g\big)=\sum_{\{x,y\}\in E_B}\varphi_\xy|g(x)-g(y)|^2.
\end{equation}
This is a formal application of the LDP for the normalized occupation times of a Markov process with symmetric generator $\Delta^\varphi$ as in \cite{DV75} and \cite{G77}; by $(\cdot,\cdot)$ we denote the standard inner product on $\ell^2(\Z^d)$. Note that the event $\{X_{[0,t]}\subset B\}$ is contained in $\{\frac{1}{t}\ell_t\approx g^2\}$, therefore we drop it from the notation. 

Taking random conductances into account, we expect an LDP on a slower scale than $t$, as small $t$-dependent values of the conductances lead to a slower decay of the annealed probability of the event $\{\frac{1}{t}\ell_t\approx g^2\}$. Therefore, we rescale $\om$ by a factor $t^r$ with some $r>0$ to be determined later, and approximate
\begin{align}\label{eqn:LD_omega_finite}
\Pr\big(t^r\om\approx \varphi\big)&=\Pr\big(\forall \{x,y\}\in E_B\colon \om_\xy\approx t^{-r}\varphi_\xy\big)=\prod_{\{x,y\}\in E_B}\Pr\big(\om_\xy\approx t^{-r}\varphi_\xy\big)\notag\\
&\approx\exp\big\{-t^{r\eta}H(\varphi)\big\},
\end{align}
where the rate function for the conductances  is given by
\begin{equation}\label{Hdef}
H(\varphi):=D\sum_{\{x,y\}\in E_B}\varphi_\xy^{-\eta}.
\end{equation} 
Here we made use of the tail assumptions in \eqref{eqn:tails_behaviour}. Hence, combining \eqref{eqn:LD_inthebox_finite} and \eqref{eqn:LD_omega_finite},
\begin{align}\label{eqn:LD_combined_finite}
\Big\langle\P_0^{\om}\big(\smfrac{1}{t}\ell_t\approx g^2\big)\1_{\{t^r\om\approx\varphi\}}\Big\rangle
&\approx\P_0^{t^{-r}\varphi}\big(\smfrac{1}{t}\ell_t\approx g^2\big)\Pr\big(\om\approx t^{-r}\varphi\big)\notag\\
&\approx \exp\Big\{- tI_{t^{-r}\varphi}(g^2)-t^{r\eta}H(\varphi)\Big\}\notag\\
&\approx\exp\Big\{-\sum_{\{x,y\}\in E_B} \Big(t^{1-r}\varphi_\xy\big(g(x)-g(y)\big)^2+t^{r\eta}D\varphi_\xy^{-\eta}\Big)\Big\}.
\end{align}
We obtain the slowest decay by choosing $r$ such that $t^{1-r}=t^{r\eta}$, which means $r=(1+\eta)^{-1}$. Then the right-hand side has scale $t^{\frac{\eta}{\eta+1}}$, which is the scale of the desired LDP. In order to find the rate function, we optimize over $\varphi$ and obtain that the choice $\varphi=\varphi^{\ssup g}$ with
\begin{equation}\label{eqn:shape_order_finite}
\varphi^{\ssup g}_\xy=(D\eta)^\frac{1}{\eta+1}|g(y)-g(x)|^{-\frac{2}{\eta+1}},\qquad \{x,y\}\in E_B,
\end{equation}
contributes most to the joint probability.
Therefore, we have the result
\begin{equation*}
\Big\langle\P_0^{\om}\big(\smfrac{1}{t}\ell_t\approx g^2\big)\Big\rangle\approx\exp\Big\{-t^{\frac{\eta}{\eta+1}} J(g^2)\Big\},
\end{equation*}
where the rate function is identified as
\begin{equation}\label{rem:rate_function_opt}
 J(g^2)=\inf_\varphi\big[ I_\varphi(g^2) + H(\varphi)\big]=I_{\varphi^{\ssup g}}(g^2) + H(\varphi^{\ssup g})=K_{\eta,D}\sum_{\{x,y\}\in E_B}|g(y)-g(x)|^\frac{2\eta}{\eta+1}.
\end{equation}

The tail assumptions we have made on the environment distribution lead to a fairly remarkable interaction between the random influences of the environment on the one hand and the random walk on the other. Under more general assumptions, e.g.,
\begin{equation*}
\log\Pr(\om_\xy\leq\eps)\sim -\alpha(\eps),\qquad\eps\to 0
\end{equation*}
for some sufficiently regular nonincreasing function $\alpha\colon\R_+\to\R_+$, we would expect an analogous result to hold. However, if $\alpha(\eps)$ is not a polynomial in $\eps$, the scale and rate function of a corresponding LDP certainly would not have such an explicit form.

\section{Proof of Theorem~\ref{thm:LDP_finite}}\label{sec:proof}

In this section, we prove Theorem \ref{thm:LDP_finite}. This amounts to showing the two inequalities in \eqref{eqn:lbound} and \eqref{eqn:ubound}, since the compactness of the level sets follows immediately from the continuity of $J$ and compactness of the space $\Mcal_1(B)$. The two inequalities are proven in the next two sections.

\subsection{Proof of the lower bound}\label{sec:proof_lb}

In order to prove \eqref{eqn:lbound}, we need to control the transition from one realization of the environment to another. To this end, we first identify the density of this transition on process level. We feel that this should be generally known, but could not find a suitable reference. For $\varphi\colon E\to(0,\infty)$ we abbreviate $\bar\varphi(x):=\sum_{y\sim x}\varphi(x,y)$. We also write $\varphi_{\xy}$ instead of $\varphi(x,y)$.

\begin{lemma}\label{lem:density}
Assume that $\varphi,\psi\colon E\to(0,\infty)$ are bounded both from above and away from zero. Denote by $S(t)$ the number of jumps the process $X=(X_s)_{s\in[0,t]}$ makes up to time $t$ and by $0<\tau_1<\ldots<\tau_{S(t)}$ the corresponding jump times. Fix some starting point $x\in\Z^d$ and put $\tau_0=0$. Then, for all $t\in[0,\infty)$,
\begin{equation*}
\Phi_t(X):=
\prod_{i=1}^{S(t)}
\left(\frac{\varphi(X_{\tau_{i-1}},X_{\tau_{i}})}{\psi(X_{\tau_{i-1}},X_{\tau_{i}})}
\e^{-(\tau_i-\tau_{i-1})\left[\bar\varphi(X_{\tau_{i-1}})-\bar\psi(X_{\tau_{i-1}})\right]}
\right)
\e^{-(t-\tau_{S(t)})\left[\bar\varphi(X_t)-\bar\psi(X_t)\right]}
\end{equation*}
is the Radon-Nikodym density of $\P_x^\varphi$ with respect to $\P_x^\psi$ with time horizon $t$. 
\end{lemma}

\begin{proof}
We will write $\Phi_t$ instead of $\Phi_t(X)$. Obviously, $\Phi_t>0$ almost surely. We start showing that, for all $t\geq0$, the expectation of $\Phi_t$ under $\P_x^\psi$ is one. Then, we use Kolmogorov's extension theorem to show the existence of a measure $\P_x$ such that $\P_x(A)=\me_x^\psi(\Phi_t\1_A)$ for all $A\in\Fcal_t$, where $(\Fcal_t)_{t\in[0,\infty)}$ is the natural filtration generated by $X$. It remains to show that the process $X$ under $\P_x$ is a Markov process and that it is generated by $\Delta^\varphi$, which implies $\P_x=\P_x^\varphi$.

Let us start by showing that the expectation of $\Phi_t$ under $\P_x^\psi$ is one. Consider the discrete-time process
\begin{equation*}
Z_n:=\prod_{i=1}^{n}\left(\frac{\varphi(X_{\tau_{i-1}},X_{\tau_{i}})}{\psi(X_{\tau_{i-1}},X_{\tau_{i}})}
\e^{-(\tau_i-\tau_{i-1})\left[\bar\varphi(X_{\tau_{i-1}})-\bar\psi(X_{\tau_{i-1}})\right]}
\right).
\end{equation*}
We have, for $x\in\Z^d$,
\begin{equation*}
\me^\psi_x[Z_1]=\sum_{y\sim x}\frac{\psi_\xy}{\bar\psi(x)}\frac{\varphi_\xy}{\psi_\xy}\int_0^\infty
\,\bar\psi(x)\e^{-\bar\psi(x)s-(\bar\varphi(x)-\bar\psi(x))s}\,{\rm d}s=\sum_{y\sim x}\frac{\varphi_\xy}{\bar\varphi(x)}=1.
\end{equation*}
Combining this equation with the strong Markov property, we see that $(Z_n)_n$ is a martingale with respect to the filtration $(\Fcal_{\tau_n})_{n\in\N}$ generated by the jumping times and that 
\begin{equation}\label{ez1}
\me^\psi_x\left[\frac{\varphi(X_t,X_{\tau_{S(t)+1}})}{\psi(X_t,X_{\tau_{S(t)+1}})}
\e^{-(\tau_{S(t)+1}-t)\left[\bar\varphi(X_t)-\bar\psi(X_t)\right]}\Big\vert
\Fcal_t\right]=\me^\psi_{X_t}\left[Z_1\right]=1
\end{equation}
$\P_x^\psi$-almost surely for all $x\in\Z^d$. Then, we obtain
\begin{equation*}
\me^\psi_x[\Phi_t]=\me^\psi_x[Z_{S(t)+1}],\qquad x\in\Z^d,
\end{equation*}
by inserting the first term of (\ref{ez1}) under the expectation and using that $\Phi_t$ is $\Fcal_t$-measurable. Consequently, it remains to show that $\me^\psi_x[Z_{S(t)+1}]=1$. As $S(t)+1$ is an unbounded, but almost surely finite stopping time with respect to the filtration $(\Fcal_{\tau_n})_{n\in\N}$, the optional sampling theorem yields that $\me^\psi_x[Z_{S(t)+1}]\leq 1$. On the other hand, for all integers $k>0$,
\begin{equation}\label{ez2}
\me^\psi_x[Z_{S(t)+1}]
\geq\me^\psi_x[Z_{S(t)+1}\1_{S(t)+1\leq k}]
=\me^\psi_x[Z_{S(t)+1\wedge k}]-\me^\psi_x[Z_k\1_{S(t)
\geq k}]=1-\me^\psi_x[Z_k\1_{S(t)\geq k}].
\end{equation}
To show that the last term is arbitrarily close to one for large $k$, we recall that on $\{S(t)\geq k\}$
\begin{equation*}
Z_k\leq\left(\frac{\max_{x\in\Z^d,\,y\sim x} \varphi_\xy}{\min_{x\in\Z^d,\,y\sim x} \psi_\xy}\right)^k \e^{t \max\left\{|\varphi_\xy-\psi_\xy|\colon \{x,y\}\in E\right\}}=:\alpha_k,
\end{equation*}
so $\me^\psi_x[Z_k\1_{S(t)\geq k}]$ is bounded from above by $\alpha_k\P_x^\psi(S(t)\geq k)$. As all jumping times are exponentially distributed with a parameter smaller than $\gamma:=\max_{x\in\Z^d}\bar\psi(x)$, we may estimate
\begin{equation*}
\P_x^\psi(S(t)\geq k)\leq\e^{\gamma t}\sum_{n=k}^\infty\frac{(\gamma t)^n}{n!}.
\end{equation*}
The tail of an exponential series is super-exponentially small, which means $\alpha_k\P_x^\psi(S(t)\geq k)\to 0$ for $k\to\infty$. Since (\ref{ez2}) was true for all $k$, we see that $\me^\psi_x[Z_{S(t)+1}]=1$.

For arbitrary $k\in\N$ and $t_1,\ldots,t_k\geq 0$ define $\hat t=\max_{i\in\{1,\ldots,k\}}t_i$ and a measure $Q_{t_1,\ldots,t_k}$ on ${(\Z^d)}^k$ by 
$$
Q_{t_1,\ldots,t_k}(x_1,\ldots,x_k)=\me_x^\psi[\Phi_{\hat t}\1_{\{X_{t_1}=x_1,\ldots,X_{t_k}=x_k\}}],\qquad x_1,\ldots,x_k\in\Z^d.
$$

We verify without much effort that $\me_x^\psi[\Phi_{t+s}\1_A]=\me_x^\psi[\Phi_t\1_A]$ for all $A\in\Fcal_t$ and $t,s>0$, which implies consistency of the family of measures above. Thus, by Kolmogorov's extension theorem, there exists a measure $\P_x$ with finite-dimensional distributions as above, and we have $\P_x(A)=\me_x^\psi[\Phi_t\1_A]$ for all $t>0$ and $A\in\Fcal_t$. We show that the process $X$ under $\P_x$ satisfies the Markov property, i.e.,
\begin{equation}\label{eqn:MarkovProperty}
\me_x[\1_{\{X_{t+s}=y\}}|\Fcal_t]=\P_{X_t}(X_s=y)\quad \P_x\text{-a.s. for all } y\in\Z^d, s,t>0 
\end{equation}
where $\me_x$ denotes expectation with regard to $\P_x$. Note that $\P_{X_t}$ is defined as we have considered an arbitrary starting point $x$ in what we have shown so far. Indeed, for all $A\in\Fcal_t$
\begin{align*}
\me_x\big[\me_x[\1_{\{X_{t+s}=y\}}|\Fcal_t]\1_A\big]&=\me_x[\1_{\{X_{t+s}=y\}}\1_A]=\me^\psi_x[\Phi_{t+s}\1_{\{X_{t+s}=y\}}\1_A]\\
&=\me^\psi_x\big[\me^\psi_x[\Phi_{t+s}\1_{\{X_{t+s}=y\}}|\Fcal_t]\1_A\big]\\
&\overset{(\ast)}=\me^\psi_x\big[\Phi_t\me^\psi_{X_t}[\Phi_s\1_{\{X_s=y\}}]\1_A\big]\\
&=\me_x\big[\me_{X_t}[\1_{\{X_s=y\}}]\1_A\big],
\end{align*}
where equation ($*$) is due to the fact that $X$ satisfies the Markov property under $\P_x^\psi$ and $\Phi_{t+s}\Phi_t^{-1}\1_{\{X_{t+s}=y\}}$ depends only on $X_{[t,t+s]}$. Consequently, we have shown \eqref{eqn:MarkovProperty} and $X$ is a Markov process under $\P_x$ with a unique infinitesimal generator. Elementary calculations show that
\begin{align*}
\frac 1t\Big(\me^\psi_x\left[f(X_t)\Phi_t\right]-f(x)\Big)\xrightarrow{t\to 0}\Delta^\varphi f(x)
\end{align*}
for arbitrary $x\in\Z^d$ and $f\colon\Z^d\rightarrow\R$. This implies $\P_x=\P_x^\varphi$ and the proof is complete.
\end{proof}

Now we use Lemma~\ref{lem:density} to compare probabilities for two environments that are close to each other.

\begin{cor}\label{cor:LDP_density_estimate}
Let $\varphi,\psi\colon E\to(0,\infty)$ with $0<\psi_\xy-\eps\leq\varphi_\xy\leq\psi_\xy+\eps$ for some $\eps>0$ and all $\{x,y\}\in E$. Moreover, let $F$ be some event that depends on the process $(X_s)_{s\in[0,t]}$ up to time $t$ only. Then
$$
\P_0^\varphi\big(F\big)\geq\e^{-4d\eps t}\P_0^{\psi-\eps}\big(F\big).
$$
\end{cor}

\begin{proof}
Let $\Phi_t$ denote the Radon-Nikodym density of $\P_0^\varphi$ with respect to $\P_0^{\psi-\eps}$ up to time $t$. Employing the representation given in Lemma~\ref{lem:density}, we have
\begin{align*}
\Phi_t&\geq\prod_{i=1}^{S(t)}\left(\e^{-(\tau_i-\tau_{i-1})\left[\bar\varphi(X_{\tau_{i-1}})-\bar\psi(X_{\tau_{i-1}})+2d\eps\right]}\right)\e^{-(t-\tau_{S(t)})\left[\bar\varphi(X_t)-\bar\psi(X_t)+2d\eps\right]}\\
&\geq\prod_{i=1}^{S(t)}\left(\e^{-(\tau_i-\tau_{i-1})4d\eps}\right)\e^{-(t-\tau_{S(t)})4d\eps}\geq\e^{-4d\eps t}.
\end{align*}
The desired inequality follows immediately.
\end{proof}
\begin{remark}\label{rem:LDP_density_estimate}
If the event $A$ is contained in $\{\supp(\ell_t)\subset B\}$, it suffices to require $0<\psi_\xy-\eps\leq\varphi_\xy\leq\psi_\xy+\eps$ for some $\eps>0$ and all $\{x,y\}\in E_B$.
\end{remark}

Let us now show (\ref{eqn:lbound}). Fix an open set $O\subset\Mcal_1(B)$. As the event $\{X_{[0,t]}\subset B\}$ is contained in $\{\frac1t\ell_t\in O\}$, we omit it in the notation. Observe that the distributions of $\frac{1}{t}\ell_t$ under $\P_0^\om$ and $\frac{1}{t^{1-r}}\ell_{t^{1-r}}$ under $\P_0^{t^r\om}$ coincide for all $0<r<1$. Hence
$$
\liminf_{t\to\infty}\frac 1{t^{\frac{\eta}{\eta+1}}}\log\Big\langle\P_0^{\om}\Big(\smfrac{1}{t}\ell_t\in O\Big)\Big\rangle
=\liminf_{t\to\infty}\frac1t\log\Big\langle\P_0^{t^\frac{1}{\eta}\om}\Big(\smfrac{1}{t}\ell_t\in O\Big)\Big\rangle,
$$
which will simplify the application of a classical Donsker-Varadhan LDP for random walks in fixed environment later. Choose an element $g^2\in O$ arbitrarily. For $M>0$ define $ \varphi^{\ssup g}_M\colon E_B\to(0,\infty)$ by
$$
\varphi^{\ssup g}_M(x,y)=\begin{cases}
(D\eta)^\frac{1}{\eta+1}|g(y)-g(x)|^{-\frac{2}{\eta+1}}&\mbox{if }|g(y)-g(x)|>0,\\
M&\mbox{otherwise.} 
\end{cases}
$$
Next, we introduce the set
\begin{equation}
A=\big\{\varphi\colon E_B\to(0,\infty)\,\big|\,\varphi^{\ssup g}_M-\eps\leq\varphi\leq \varphi^{\ssup g}_M\big\},
\end{equation}
where $\eps>0$ is picked smaller than $\frac 12\min_{E_B}\varphi^{\ssup g}_M$. By dint of Corollary \ref{cor:LDP_density_estimate},
\begin{align}\label{LDP_LB1_finite}
\Big\langle\P_0^{t^\frac{1}{\eta}\om}\Big(\smfrac 1t\ell_t\in O\Big)\Big\rangle
&\geq\Big\langle\P_0^{t^\frac{1}{\eta}\om}\Big(\smfrac 1t\ell_t\in O\Big)\1_{\big\{t^\frac{1}{\eta}\om\in A\big\}}\Big\rangle\notag\\
&\geq\inf_{\varphi\in A}\P_0^{\varphi}\Big(\smfrac 1t\ell_t\in O\Big)\Pr\big(t^\frac{1}{\eta}\om\in A)\notag\\
&\geq\e^{-4d\eps t}\P_0^{\varphi_M^{\ssup g}-\eps}\Big(\smfrac 1t\ell_t\in O\Big)\Pr\big(t^\frac{1}{\eta}\om\in A).
\end{align}
Using the tail assumption in \eqref{eqn:tails_behaviour}, we see that
$$
\lim_{t\to\infty}\frac 1t\log\Pr\big(t^\frac{1}{\eta}\om\in A)=-H(\varphi^{\ssup g}_M),
$$
where $H$ is given in \eqref{Hdef}. Furthermore, we apply the lower bound of the classical Donsker-Varadhan LDP (see \cite{DV75} or \cite{G77}) to get
$$
\liminf_{t\to\infty}\frac 1t\log \P_0^{\varphi_M^{\ssup g}-\eps}\Big(\smfrac 1t\ell_t\in O\Big)\geq -\inf_O I_{\varphi_M^{\ssup g}-\eps},
$$
where $I_\varphi$ is given in \eqref{Idef}. Hence, from \eqref{LDP_LB1_finite} we obtain
\begin{equation*}
\begin{aligned}
\liminf_{t\to\infty}\frac 1t\log\Big\langle\P_0^{t^\frac{1}{\eta}\om}\Big(\smfrac{1}{t}\ell_t\in O\Big)\Big\rangle
&\geq-4d\eps-\inf_{O}I_{\varphi_M^{\ssup g}-\eps}-H(\varphi^{\ssup g}_M)\\
&\geq-4d\eps-\inf_{O}I_{\varphi^{\ssup g}_M}-H(\varphi^{\ssup g}_M)\\
&\geq-4d\eps-I_{\varphi^{\ssup g}_M}(g^2)-H(\varphi^{\ssup g}_M),
\end{aligned}
\end{equation*}
since $I_{\varphi_M^{\ssup g}-\eps}\leq I_{\varphi^{\ssup g}_M}$ and $g^2\in O$. Now we send $\eps$ to zero and $M$ to $\infty$, to obtain 
\begin{equation*}
\liminf_{t\to\infty}\frac 1t\log\Big\langle\P_0^{t^\frac{1}{\eta}\om}\Big(\smfrac{1}{t}\ell_t\in O\Big)\Big\rangle\geq -I_{\varphi^{\ssup g}}(g^2)-H(\varphi{\ssup g})=-J(g^2),
\end{equation*}
where $\varphi^{\ssup g}=\lim_{M\to\infty}\varphi^{\ssup g}_M$ is given in \eqref{eqn:shape_order_finite}, and we used \eqref{rem:rate_function_opt}. The desired lower bound follows by passing to the infimum over all $g^2\in O$.

\subsection{Proof of the upper bound}\label{sec:proof_ub}

In this section we prove \eqref{eqn:ubound}. Let us first fix some configuration $\varphi\in(0,\infty)^E$ and start with an estimate for the probability $\P^\varphi_0(\frac{1}{t}\ell_t\in\cdot)$. This approach has actually been used by other authors before, but we provide an independent proof for the sake of completeness.

\begin{lemma}\label{lemma:ub_via_fc}
Fix an arbitrary set $A\subset\Mcal_1(B)$. Then
\begin{equation}
\P^\varphi_0\Big(\smfrac{1}{t}\ell_t\in A\Big)\leq\frac{f(0)}{\min_B f}\exp\Big\{ t\sup_{h^2\in A}\sum_{x\in B}\frac{\Delta^\varphi f(x)}{f(x)}h^2(x)\Big\}
\end{equation}
for arbitrary $f\colon\Z^d\to [0,\infty)$ with $\supp (f)= B$ and $t>0$.
\end{lemma}
\begin{proof}
We consider the Cauchy problem
\begin{equation}\label{eqn:cauchy_problem}
\begin{cases}
\,\partial_t u(x,t)=\Delta^\varphi u(x,t)+V(x)u(x,t),\quad&x\in\Z^d,\,t>0,\\
\,u(x,0)=f(x),\quad&x\in\Z^d,
\end{cases} 
\end{equation}
with
$$
V=-\frac{\Delta^\varphi f}{f}\1_B.
$$
Obviously, $u(\cdot,t)\equiv f(\cdot)$ solves (\ref{eqn:cauchy_problem}). On the other hand, by the Feynman-Kac formula, any nonnegative solution $u$ satisfies
\begin{equation}\label{eqn:fc_formula}
 u(x,t)=\E_x^\varphi\Big[\e^{\int_0^tV(X_s){\rm d}s}u(X_t,t)\Big],\qquad x\in\Z^d,\,t\geq0.
\end{equation}
Therefore, we may estimate
\begin{align*}
 f(0)&=\E_0^\varphi\Big[\e^{-\int_0^t\frac{\Delta^\varphi f(X_s)}{f(X_s)}{\rm d}s}f(X_t)\Big]
\\
&\geq \E_0^\varphi\Big[\e^{-\sum_{x\in B}\frac{\Delta^\varphi f(x)}{f(x)}\ell_t(x)}f(X_t)\1_{\{\frac{1}{t}\ell_t\in A\}}\Big]
\\ 
&\geq \min_B f\,\exp\Big\{-t\sup_{h^2\in A}\sum_{x\in B}\frac{\Delta^\varphi f(x)}{f(x)}h^2(x)\Big\} \P_0^\varphi\Big(\smfrac{1}{t}\ell_t\in A\Big),
\end{align*}
which is a rearrangement of the assertion.
\end{proof}

Now fix some closed set $C\subset\Mcal_1(B)$. As a closed subset of a finite-dimensional space, $C$ is compact with respect to the Euclidean topology. We are going to apply a standard compactness argument, which is in the spirit of the proof of the upper bound in Varadhan's lemma \cite[Thm.~4.3.1]{DZ98}. The idea is to cover $C$ with certain open balls, where \lq open\rq\ refers to the Euclidean topology.

Fix $\delta >0$. For $g^2\in C$ define
$$ 
d_g= \min\big\{| g(y)- g(x)|\colon\{x,y\}\in E,\,g(x)\not=g(y)\big\}\in(0,\infty),
$$
where we recall that $g^2$ is defined on the entire $\Z^d$ and is zero outside $B$. Consider the open ball in $\Mcal_1(B)$ of radius $\delta_g:=\min\{ d_g^4, \delta\}$ centered at $g^2$. Fixing a configuration $\varphi\in(0,\infty)^E$, we can apply Lemma \ref{lemma:ub_via_fc} with $f(\cdot):= g(\cdot)+\sqrt{\delta_g}\1_B$ and obtain
\begin{align}\label{eqn:main_estimate}
\P^\varphi_0\Big(\smfrac{1}{t}\ell_t\in B_{\delta_g}(g^2)\Big)&\leq \frac{1+\sqrt{\delta_g}}{\sqrt{\delta_g}}\exp\Big\{t\sup_{h^2\in B_{\delta_g}( g^2)}\sum_{x\in B}\frac{\Delta^\varphi ( g+\sqrt{\delta_g}\1_{B})(x)}{ g(x)+\sqrt {\delta_g}}h^2(x)\Big\}.
\end{align}

In what follows, we show
\begin{align}\label{eqn:ball_estimate}
\sup_{h^2\in B_{\delta_g}(g^2)}\sum_{x\in B}\frac{\Delta^\varphi ( g+\sqrt \delta_g\1_{B})(x)}{ g(x)+\sqrt \delta_g}h^2(x)&\leq-I_\varphi(g^2)(1-7\delta^{\frac{1}{4}}),
\end{align}
where we recall from \eqref{Idef} that $I_\varphi(g^2)=\sum_{\{x,y\}\in E}\varphi_{xy}|g(x)-g(y)|^2=-(\Delta^\varphi g,g)$.
%
%
%
To that end, we replace $h^2$ by $(g+\sqrt{\delta_g}\1_B)^2$ and control the error terms.
\begin{align}\label{eqn:exponent}
\sup_{h^2\in B_{\delta_g}(g^2)}&\sum_{x\in B}\frac{\Delta^\varphi ( g+\sqrt{\delta_g}\1_{B})(x)}{ g(x)+\sqrt{\delta_g}}h^2(x)\nonumber\\
=&\sum_{x\in B}\frac{\Delta^\varphi ( g+\sqrt{\delta_g}\1_{B})(x)}{ g(x)+\sqrt{\delta_g}}(g(x)+\sqrt{\delta_g})^2 \nonumber\\
&+\sup_{h^2\in B_{\delta_g}(g^2)}\sum_{x\in B}\frac{\Delta^\varphi ( g+\sqrt{\delta_g}\1_{B})(x)}{ g(x)+\sqrt{\delta_g}}\big[(h^2(x)-g^2(x))-2\sqrt{\delta_g}g(x)-\delta_g\big].
\end{align}

The first sum is easily estimated against the standard Donsker-Varadan rate function:
$$
\begin{aligned}
\sum_{x\in B}\frac{\Delta^\varphi(g+\sqrt{\delta_g}\1_{B})(x)}{g(x)+\sqrt{\delta_g}}(g(x)+\sqrt{\delta_g})^2
&=\big(\Delta^\varphi(g+\sqrt{\delta_g}\1_{B}),g+\sqrt{\delta_g}\1_{B}\big)\\
&\leq \big(\Delta^\varphi g,g\big)=-I_\varphi(g^2),
\end{aligned}
$$
where we have used the symmetry of the operator $\Delta^\varphi$ and that $g=0$ outside $B$. In order to estimate the last term in (\ref{eqn:exponent}), we treat the contribution of every summand within the square brackets separately. We begin with the first part and observe that $|h^2(x)- g^2(x)|=|h(x)-g(x)|\,|h(x)+g(x)|\leq 2\delta_g$ for all $h^2\in B_{\delta_g}(g^2)$ and $x\in B$. Thus
\begin{align*}
\sum_{x\in B}&\frac{\Delta^\varphi ( g+\sqrt{\delta_g}\1_{B})(x)}{ g(x)+\sqrt{\delta_g}}(h^2(x)-g^2(x))\\
&=\sum_{\substack{\{x,y\}\in E\colon\\ x,y\in B}} \varphi_{xy}\frac{g(y)-g(x)}{g(x)+\sqrt{\delta_g}}(h^2(x)-g^2(x))-
\sum_{\substack{\{x,y\}\in E:\\ x\in B, y\not\in B}}\varphi_{xy}(h^2(x)-g^2(x))\\
&\leq \sum_{\substack{\{x,y\}\in E\\ x,y\in B}} \varphi_{xy} \frac{|g(x)-g(y)|}{\sqrt{\delta_g}}2\delta_g +\sum_{\substack{\{x,y\}\in E:\\ x\in B, y\not\in B}}\varphi_{xy}2\delta_g\\
&\leq 4\delta^{\frac{1}{4}}I_\varphi(g^2).
\end{align*}
The last step is due to the fact that $\delta_g^{\frac{1}{4}}\leq g(x)-g(y)$ whenever $g(x)-g(y)>0$. Secondly,
\begin{align*}
\sum_{x\in B}&\frac{\Delta^\varphi ( g+\sqrt \delta_g\1_{B})(x)}{ g(x)+\sqrt{\delta_g}}(-2\sqrt{\delta_g}g(x))\\
&\leq\sum_{\substack{\{x,y\}\in E\colon\\x,y\in B}} \varphi_{xy}|g(x)-g(y)|
\Big|\frac{2\sqrt{\delta_g}g(x)}{g(x)+{\sqrt{\delta_g}}}-\frac{2\sqrt{\delta_g}g(y)}{g(y)+\sqrt{\delta_g}}\Big|+
\sum_{\substack{\{x,y\}\in E\colon\\ x\in B,y\not\in B}}\varphi_{xy}2\sqrt{\delta_g}g(x)\\
&\leq\sum_{\substack{\{x,y\}\in E\colon\\x,y\in B}} \varphi_{xy}|g(x)-g(y)|^2 \frac{2\delta_g}{\sqrt{\delta_g}d_g}+
\sum_{\substack{\{x,y\}\in E\colon\\ x\in B, y\not\in B}}\varphi_{xy}2\sqrt{\delta_g}|g(x)-g(y)|\\
&\leq 2\delta^{\frac{1}{4}}I_\varphi(g^2).
\end{align*}
Here, we have used $\delta_g^{\frac{1}{4}}\leq d_g$. The only part left is
\begin{align*}
\sum_{x\in B}&\frac{\Delta^\varphi ( g+\sqrt{\delta_g}\1_{B})(x)}{ g(x)+\sqrt{\delta_g}}(-\delta_g)\\
&\leq\sum_{\substack{\{x,y\}\in E\colon\\x,y\in B}} \varphi_{xy}|g(x)-g(y)|
\Big|\frac{1}{g(x)+{\sqrt{\delta_g}}}-\frac{1}{g(y)+\sqrt{\delta_g}}\Big|\delta_g
+\sum_{\substack{\{x,y\}\in E\colon\\ x\in B,y\not\in B}}\varphi_{xy}\delta_g\\
&\leq\sum_{\substack{\{x,y\}\in E\colon\\x,y\in B}} \varphi_{xy}|g(x)-g(y)|^2
\frac{1}{\sqrt{\delta_g}d_g}\delta_g
+\sum_{\substack{\{x,y\}\in E\colon\\ x\in B,y\not\in B}}\varphi_{xy}\delta_g\\
&\leq \delta^{\frac{1}{4}}I_\varphi(g^2).
\end{align*}
Combining (\ref{eqn:exponent}) with the last three estimates, we obtain (\ref{eqn:ball_estimate}) and in particular
\begin{equation}\label{eqn:estifinal}
\P_0^\varphi\Big(\smfrac{1}{t}\ell_t\in B_{\delta}( g^2)\Big)\leq\frac{1+\sqrt{\delta_g}}{\sqrt{\delta_g}}\prod_{\{x,y\}\in E}\exp\big\{-t\,\varphi_{xy}|g(x)-g(y)|^2(1-7\delta^{\frac{1}{4}})\big\}.
\end{equation}
The balls $B_{\delta_g}(g^2)$ with $g^2\in C$ cover $C$ and since this set is compact, we may extract a finite subcovering of $C$. Denote by $(g^2_i)_{i=1,\dots,N}$ the centers of the balls in this subcovering. Then, applying \eqref{eqn:estifinal} for $\varphi=t^{\frac{1}{\eta}}\omega$, we obtain
\begin{align*}
\limsup_{t\to\infty}&\frac{1}{t}\log \Big\langle \P_0^{t^{\frac{1}{\eta}}\omega}\Big(\smfrac{1}{t}\ell_t\in C\Big)\Big\rangle\\ 
\leq&\max_{i=1,\dots,N}\limsup_{t\to\infty} \frac{1}{t}\log \Big\langle \P_0^{t^{\frac{1}{\eta}}\omega}\Big(\smfrac{1}{t}\ell_t\in B_{\delta_{g_i}}(g_i^2)\Big)\Big\rangle\\
\leq &\max_{i=1,\dots,N}\sum_{\{x,y\}\in E_B}\limsup_{t\to\infty} \frac{1}{t}\log\Big\langle  \exp\big\{-t^{\frac{1+\eta}{\eta}}\omega_{xy}|g_i(y)-g_i(x)|^2(1-7\delta^{\frac{1}{4}})\big\} 
\Big\rangle.
\end{align*}

According to de Bruijn's exponential Tauberian theorem \cite[Theorem  4.12.9]{BGT89}, the tail assumption (\ref{eqn:tails_behaviour}) is equivalent to the condition that, for any $M>0$ and $\{x,y\}\in E$,
\begin{equation}\label{eqn:TailEquivalent}
\lim_{t\to\infty}\frac{1}{t}\log\Big\langle \exp\big\{-t^{\frac{1+\eta}{\eta}}\omega_\xy M\big\}\Big\rangle=
-K_{\eta,D} M^{\frac{\eta}{1+\eta}},
\end{equation}
where we recall $K_{\eta,D}=\big(1+\frac{1}{\eta}\big)(D\eta)^\frac{1}{\eta+1}$ from Theorem~\ref{thm:LDP_finite}.
Thus, with $\delta$ so small that $1-7\delta^{\frac{1}{4}}>0$, we obtain
\begin{align*}
\limsup_{t\to\infty} \frac{1}{t}\log \Big\langle \P^{t^{\frac{1}{\eta}}\omega}\Big(\smfrac{1}{t}\ell_t\in C\Big)\Big\rangle
&\leq 
\max_{i=1,\dots,N}\sum_{\{x,y\}\in E_B}-K_{\eta,D} |g_i(y)-g_i(x)|^{\frac{2\eta}{1+\eta}}(1-7\delta^{\frac{1}{4}})^{\frac{\eta}{1+\eta}}
\\
&\leq -(1-7\delta^{\frac{1}{4}})^{\frac{\eta}{1+\eta}}\inf_{g^2\in C}J(g^2)
\end{align*}
with $J$ as in \eqref{rem:rate_function_opt}. Since we may choose $\delta$ arbitrarily small, the proof of \eqref{eqn:ubound} is complete.


\end{document}